\documentclass{article}
\usepackage{graphicx} 

\usepackage[utf8]{inputenc}
\usepackage{geometry}
\usepackage{amsmath, amssymb}
\usepackage{algorithm}[boxed,boxedruled]
\usepackage[algo2e]{algorithm2e}[linesnumbered,vlined]
\usepackage{mathrsfs}
\usepackage{amsthm}
\usepackage{float}
\usepackage{graphicx}
\usepackage{caption}
\usepackage{subcaption}
\usepackage{authblk}
\usepackage{orcidlink}

\newtheorem{theorem}{Theorem}[section]
\newtheorem*{theorem*}{Theorem}
\newtheorem{lemma}[theorem]{Lemma}

\begin{document}

\title{Adaptive Voronoi-based Column Selection Methods for Interpretable Dimensionality Reduction}

\author[a,\orcidlink{0000-0002-3209-4705}]{Maria Emelianenko \footnote{E-mail: memelian@gmu.edu}}
\author[a,\orcidlink{guyoldaker4@gmail.com}]{Guy B. Oldaker IV\footnote{E-mail: goldaker@gmu.edu} }

\affil[a]{Department of Mathematical Sciences, George Mason University, 4400 University Dr, Fairfax, VA 22030}

\maketitle


\begin{abstract}
In data analysis, there continues to be a need for interpretable dimensionality reduction methods whereby instrinic meaning associated with the data is retained in the reduced space.  Standard approaches such as Principal Component Analysis (PCA) and the Singular Value Decomposition (SVD) fail at this task.  A popular alternative is the CUR decomposition. In an SVD-like manner, the CUR decomposition approximates a matrix $A \in \mathbb{R}^{m \times n}$ as $A \approx CUR$, where $C$ and $R$ are matrices whose columns and rows are selected from the original matrix \cite{goreinov1997theory}, \cite{mahoney2009cur}.  The difficulty in constructing a CUR decomposition is in determining which columns and rows to select when forming $C$ and $R$.  Current column/row selection algorithms, particularly those that rely on an SVD, become infeasible as the size of the data becomes large \cite{dong2021simpler}.  We address this problem by reducing the column/row selection problem to a collection of smaller sub-problems.  The basic idea is to first partition the rows/columns of a matrix, and then apply an existing selection algorithm on each piece; for illustration purposes we use the Discrete Empirical Interpolation Method (\textsf{DEIM}) \cite{sorensen2016deim}.  For the first task, we consider two existing algorithms that construct a Voronoi Tessellation (VT) of the rows and columns of a given matrix.  We then extend these methods to automatically adapt to the data.  The result is four data-driven row/column selection methods that are well-suited for parallelization, and compatible with nearly any existing column/row selection strategy.  Theory and numerical examples show the design to be competitive with the original \textsf{DEIM} routine.
\end{abstract}

\section{Introduction}

Machine learning and data analysis seek to distill the essential structure(s) within data.  With the prevalence of large, high-dimensional data sets, it is common practice to convert the data into a more manageable form.  A popular choice for this task is principle component analysis (PCA), where the data are projected onto a lower dimensional space that captures most of the variance \cite{hastie2009elements}, \cite{mohri2018foundations}.  This is done using a truncated form of the singular value decomposition (SVD) of the corresponding data matrix.  A drawback of this procedure is that the resulting transformed points are each linear combinations of potentially all of the singular vectors used in the projection.  Consequently, any physical meaning present in the original samples is lost \cite{mahoney2009cur}.  This phenomenon has motivated the development of tools that strive to maintain interpretability.  One such tool is the so-called CUR matrix decomposition, where a given matrix $A$ is approximated by the product of three matrices:  $C$, $U$, and $R$.  The matrices $C$ and $R$ are constructed using carefully selected columns and rows of $A$ and, in a manner reminiscent of the SVD, serve as approximate bases for the column and row spaces of $A$ respectively.  In addition, physical meaning is preserved as well as any intrinsic properties (e.g., sparsity, non-negativity).  The matrix $U$ is typically chosen to make $||A - CUR||_\gamma$ small, where $\gamma$ is usually taken to be $2$ or $F$.  

Determining the best way to select rows and columns for the CUR decomposition is an active field of research and has led to a rich array of algorithms.  These broadly range from classical approaches that rely on well-known matrix factorizations to sophisticated routines that employ data-driven probability distributions to select rows and columns.  However, the cost of many such algorithms becomes prohibitive when the size and dimension of the data are large \cite{dong2021simpler}.  In this article, we present four algorithms inspired by Voronoi Tessellation (VT) theory \cite{okabe2000spatial} that distribute the column selection task into a collection of smaller subtasks.  With the ability to be combined with any column-selection strategy, the algorithms use a data-driven approach to first select an optimal partition of the rows/columns.  We then apply an existing row/column selection algorithm to each piece of the partition and combine the results.  The Discrete Empirical Interpolation Method (\textsf{DEIM}) \cite{sorensen2016deim} is used in this article for presentation purposes.  Through parallelization, we expect the algorithms to be applicable to large, high-dimensional data sets, particularly those for which computation of an SVD is infeasible.  Theory and simulations using real and synthetic data show the algorithms to be competitive with the current state-of-the-art.  

We begin by discussing the Column-Subset Selection Problem (CSSP) \cite{boutsidis2009improved} and associated algorithms, including \textsf{DEIM}.  In the next section, we discuss the partitioning algorithms.  This includes a review of two existing approaches followed by our new modifications.  The third section introduces a post-processing routine.  This last is required to combine the \textsf{DEIM} solutions from each of the sets in the partition. With this done, we give the results of numerical experiments and a discussion of worst-case error bounds and complexity.  This is followed by a conclusion.





\section{The Column-Subset Selection Problem (CSSP)}
The primary task in constructing a CUR decomposition is the selection of the rows and columns used to form the matrices $C$ and $R$.  Since selecting rows from a matrix is equivalent to selecting columns from the transpose, we will hereafter focus on the latter. Given a matrix $A \in \mathbb{R}^{m \times n}$ with $\mbox{rank}(A) = \rho$, and a target rank, $0<r \le \rho,$ the goal of CSSP is to form $C \in \mathbb{R}^{m \times r}$ consisting of $r$ columns of $A$ that minimizes $$||(I - CC^\dagger)A||_\xi,\quad \xi \in \{2,F\},$$ over all possible $m \times r$ matrices $C$ whose columns are taken from $A$ ($C^\dagger$ denotes the Moore-Penrose pseudoinverse of the matrix $C$).

Given the inherent difficulty in solving this problem exactly \cite{shitov2017column}, most algorithms settle for a good approximation.  These can be broadly divided into two classes \cite{dong2021simpler}:  probabilistic and deterministic. Algorithms from the first class select columns using some kind of probability distribution \cite{drineas2006sampling}, \cite{drineas2006subspace}, \cite{drineas2008relative}, \cite{wang2013improving}.  For example, the algorithms in \cite{deshpande2006adaptive} select columns with probabilities proportional to their norms.  The authors in \cite{mahoney2009cur} construct so-called leverage scores.  Built using information from an SVD truncated to a desired target rank $r$, these scores allow one to sample columns that tend to have the most in common with the dominant $r$ left singular vectors. Algorithms in the second class typically use a classical matrix factorization to select columns.  For example, the LU factorization with partial pivoting (LUPP) \cite{trefethen2022numerical} and column-pivoted QR decompositions (CPQR) \cite{golub2013matrix} both select as columns the pivot elements that arise during execution.  The \textsf{DEIM} algorithm \cite{sorensen2016deim} also falls into this category.  Given a matrix $A \in \mathbb{R}^{m \times n}$ with rank $\rho$, a target rank $k \le \rho$, the \textsf{DEIM} algorithm selects columns from $A$ in an incremental fashion while simultaneously building an oblique projector.  This projector, which uses information from the top $r$ right singular vectors of $A$, removes the need to perform row operations to determine the pivot elements that will be returned by the algorithm.  

As the size of the data matrix becomes larger, some of the algorithms mentioned above become prohibitive \cite{dong2021simpler},\cite{drineas2008relative}. This is especially true for algorithms that rely on an SVD \cite{voronin2017efficient}; e.g., \textsf{DEIM} and leverage score based routines. The associated cost for computing an SVD on an $m \times n$ matrix $A$ is $\mathcal{O}(\min\{m^2n,m,n^2\})$.  For a matrix of the same size, the LUPP and CPQR routines have complexity $\mathcal{O}(m^2n)$.  These issues can be accounted for somewhat by left multiplying $A$ by a so-called random sketch, $\Gamma \in \mathbb{R}^{l \times m}$, where $l$ is set to be the desired target rank \cite{dong2021simpler}.  A popular choice for $\Gamma$ is to set the entries to be values drawn independently from a Gaussian distribution.  One then applies any of the algorithms previously discussed on the smaller matrix, $\Gamma A \in \mathbb{R}^{l \times n}$.  However, for very large data sets, one needs to take into account the construction of $\Gamma.$  

\subsection{Notation}
Before continuing, we introduce some helpful notation.  At times we will use Matlab \cite{MATLAB} notation; e.g., if $A \in \mathbb{R}^{m \times n}$ and $j$ is a positive integer, $A(:,j)$ and $A(j,:)$ denote the $j^{th}$ column and row of $A$ respectively.  Given a set $B \subset \mathbb{R}^m$ and a point $x \in \mathbb{R}^m$, we denote by $B - x$ the set 
$$B- x = \{ b - x \; | \; b \in B\}.$$
If instead $B\in \mathbb{R}^{m \times n}$ is a matrix and $x \in \mathbb{R}^m$, the matrix $B - x \in \mathbb{R}^{m \times n}$ is given by
$$B - x = [ B(:,1) - x \dots B(:,n) - x ].$$
We also let $\Omega_B$ denote the set of columns for a matrix $B$.  For a collection of matrices, $\{B_i\}_{i=1}^k$, with $B_i \in \mathbb{R}^{m \times n}$, we let $\mathsf{diag}(B_i)$ represent the $km \times kn$ matrix
$$\mathsf{diag}(B_i) = \left ( \begin{array}{ccc}
                      B_1& & \\
                       &\ddots&\\
                       & &B_k\\
                      \end{array} \right ).$$

For the vectors of all ones in $\mathbb{R}^m$ we write $\mathbf{1}_m$, and the identity matrix in $\mathbb{R}^{m \times m}$ is $I_m$.  We write $|V|$ to denote the cardinality of a set $V$, and for $a \in \mathbb{R}$, $\lfloor a \rfloor$ denotes the largest integer that does not exceed $a.$
Throughout, the data matrix under consideration will be $A \in \mathbb{R}^{m \times n}$ with $\mbox{rank}(A) = \rho$.  We let $0<k\le n$ represent the number of sets in a partition and $0<r<\rho$ be the target rank for the CSSP problem.

\section{Decomposing CSSP  via Column Partitioning}
In an effort to avoid the bottlenecks associated with the CSSP algorithms discussed above and extend their applicability to larger data sets, we propose to divide the problem into a collection of smaller sub-problems.  Given a matrix $A \in \mathbb{R}^{m \times n}$, our first task will be to determine an optimal partition of its columns.  This is accomplished by assigning the columns, $x \in \mathbb{R}^m$, of $A$ to disjoint sets, $\{V_i\}_{i=1}^k$, via the general rule:
$$x \in V_i \mbox{ if } D^2(x,z_i) < D^2(x,z_j),\quad i \neq j.$$
Here, $D: \mathbb{R}^{m} \times \mathbb{R}^m \rightarrow \mathbb{R}$ is a distortion measure (i.e., not necessarily a metric), and $z_i \in \mathbb{R}^m$ is a point associated with the set $V_i$ (the role of the $z_i$ will become clear in the following sections).  The size, $k$, of the partition is typically selected by the user, and the resulting sets $V_i$ form a Voronoi Tessellation \cite{okabe2000spatial} of the set of columns of $A$ (Hereafter, the $V_i$ will be referred to as Voronoi sets).  By varying the distortion measure, $D$, one induces different partitions.

It is by considering different distortion measures and objective functions that we arrive at four partitioning algorithms.  This diversity helps to address the fact that the definition of 'optimal partition' is problem dependent. We then apply an existing CSSP algorithm to each partition, in this case \textsf{DEIM}.  The selected columns from this last step are used to form the final result.  By combining the partitioning methods with \textsf{DEIM}, we will arrive at four column-selection algorithms, two of which are new and have the ability to adapt to the data.

\subsection{Review of CVOD}

The first partitioning algorithm we consider is the Centroidal Voronoi Orthogonal Decomposition (CVOD) \cite{du2003centroidal}.  Originally conceived as a model order reduction technique, CVOD is a generalized Centroidal Voronoi Tessellation (CVT) in which subspaces act as centroids.  Given a matrix $A \in \mathbb{R}^{m \times n}$, positive integers $r,k$, and a multi-index $d = (d_1,\ldots,d_k)^T \in \mathbb{N}^k$, define the energy functional
$$\mathcal{G}_1 = \sum_{i=1}^k \sum_{x \in V_i}||(I_m - \Theta_i)x||_s^2$$
Using the notation from above, $z_i = \Theta_ix$ and $D^2(a,b) = ||a - b||_2^2$.  The matrix $\Theta_i \in \mathbb{R}^{m \times m}$ is an orthogonal projector of $\mbox{rank}(\Theta_i) = d_i.$
CVOD seeks to solve the following:
$$\min_{\{(V_i,\Theta_i)\}_{i=1}^k}\mathcal{G}_1 \quad \mbox{such hat}$$
$$\Theta_i^2 = \Theta_i,\quad \mbox{rank}(\Theta_i) = d_i\quad i = 1,\ldots,k.$$

Minimization of $\mathcal{G}_1$ proceeds in a alternating fashion via the generalized Lloyd method \cite{du1999centroidal},\cite{du2003centroidal}, \cite{du2006convergence}.  After forming an initial partition (perhaps randomly), $\{V_i\}_{i=1}^k$, one determines the centroid for each $V_i$ by computing the matrix, $U_i\in \mathbb{R}^{m \times d_i}$, containing the top $d_i$ left singular vectors of $V_i$ for each $i$.  Once complete, the centroids are held fixed and each Voronoi set is updated by assigning each $x\in \Omega_A$ to sets via the following rule:
$$x \in V_i \iff ||(I_m - U_iU_i^T)x||_2^2 < ||(I_m - U_sU_s^T)x||_2^2\quad i \neq s.$$

Ties are broken by assigning points to the set with the smallest index.  We halt the algorithm once the difference between the energy functional $\mathcal{G}_1$ from consecutive iterations falls below a user-prescribed threshold, $\epsilon > 0$; see Algorithm \textsf{CVOD} for an overview.  Once the algorithm completes, the columns of the matrices $\{U_i\}_{i=1}^k$ can be used to form low-dimensional basis for the column space of $A$.  Applications include model order reduction, where the computed basis is paired with the Galerkin method in a fashion similar to the Proper Orthogonal Decomposition (POD) approach \cite{burkardt2006centroidal}.  We remark that the latter requires computation of a truncated SVD of the full data matrix.  CVOD, on the other hand, splits this task into more manageable sub-tasks.

\floatname{algorithm}{}
\renewcommand{\algorithmcfname}{}

\begin{algorithm}[H]
\renewcommand{\thealgorithm}{}      
\caption{\centering \textbf{Algorithm:} $\mathsf{CVOD}$}\label{alg:CVOD}
\SetAlgoRefName{\textsf{CVOD}}
\KwData{A matrix $A \in \mathbb{R}^{m \times n}$, with rank($A$) = $\rho$, a positive integer $r<\rho$, a positive integer $0<k\le m$, a multi-index of dimensions, $d = \{d_i\}_{i=1}^k$ with $\sum_{i=1}^k d_i = r$, and a positive tolerance parameter, $\epsilon$}
\KwResult{A collection, $\{V_i,U_i\}_{i=1}^{k}$, consisting of a column partitioning of $A$, and a set of lower dimensional representations of each partition. }


\vspace{1.5\baselineskip}

$\{V_i\}_{i=1}^k \leftarrow $ Randomly partition the columns of $A$\\
$j \leftarrow 1$\\
$\Delta^{j-1} \leftarrow \epsilon + 1$\\

\While{$\Delta^{j-1} > \epsilon$}{
    \hspace{0.4\baselineskip}$z_i \leftarrow 0 \in \mathbb{R}^{m},\quad i = 1,\ldots,k$\\   
    $(\{U_i\}_{i=1}^k,k) \leftarrow \mathsf{UpdateCentroidsFixed} \left (\{V_i\}_{i=1}^k, d \right )$\\    
    $ \{V_i\}_{i=1}^k \leftarrow \mathsf{FindVoronoiSets} \left (\{V_i\}_{i=1}^k,\{U_i\}_{i=1}^k,\{z_i\}_{i=1}^k\right )$\\
    
    $\mathcal{G}^{j} \leftarrow \sum_{i=1}^k\sum_{x \in V_i}||(I_m-U_iU_i^T)x||_2^2$\\
    \If{$j<2$}{
        $\Delta^{j} \leftarrow \Delta^{j-1}$\\
    }\Else{
        $\Delta^{j} \leftarrow \mathcal{G}^{j-1} - \mathcal{G}^{j}$\\
    }
    $j \leftarrow j + 1$\\
    
}
return $\{V_i,U_i\}_{i=1}^k$\\
\end{algorithm}

\subsection{Review of VQPCA}
The Vector Quantization Principal Component Analysis (VQPCA) algorithm is a nonlinear extension of PCA \cite{kambhatla1997dimension}.  It has been used in structural dynamics applications, including damage diagnosis \cite{kerschen2002non}, \cite{kerschen2005distortion}, \cite{yan2004structural}, as well as in clustering tasks for combustion simulations \cite{zdybal2022advancing}.  The corresponding energy functional is similar to that of \textsf{CVOD}:
$$\mathcal{G}_2 = \sum_{i=1}^k \sum_{x \in V_i}||\beta_i - (I_m - \Theta_i)x||_2^2.$$
The resulting objective is
$$\min_{\{(V_i,\Theta_i,\beta_i)\}_{i=1}^k} \mathcal{G}_2\quad \mbox{such that}$$
$$\Theta_i^2 = \Theta_i,\; \mbox{rank}(\Theta_i) = d_i,\; \sum_{i=1}^k d_i = r,$$ $$\beta_i \in \mathbb{R}^m,\quad \bigcup_{i=1}^k V_i = \Omega_A,\quad i = 1,\ldots,k,$$
where we assume the same input parameters as with \textsf{CVOD}.

The vectors $\beta_i$ can be found by taking partial derivatives and setting the result to zero:
$$\beta_i = \frac{1}{|V_i|} \sum_{x \in V_i}(I_m - \Theta_i)x = \bar{x}_i - \Theta_i \bar{x}_i,\quad i = 1,\ldots,k,$$  
where $\bar{x}_i = \frac{1}{|V_i|}\sum_{x \in V_i}x.$
With this in place, the solution is determined via alternating minimization.  Given some initial partition $\{V_i\}_{i=1}^k$ of the columns of $A$, we begin by determining the Voronoi set means, $\bar{x}_i$.  The centroids are found for each $i = 1,\ldots ,k$ by solving
$$\min_{\Theta_i}||(V_i - \bar{x}_i) - \Theta_i (V_i - \bar{x}_i)||_F^2 \quad \mbox{such that}$$
$$\Theta_i^2 = \Theta_i \in \mathbb{R}^{m \times m}\quad
\mbox{rank}(\Theta_i) = d_i.$$

This process is exactly the same as that for \textsf{CVOD}, only that the matrices $V_i$ have each been shifted by $\bar{x}_i$.  The resulting centroids consist of $U_i \in \mathbb{R}^{m \times d_i}$ that contain the top $d_i$ left singular vectors of each $V_i - \bar{x}_i$.   Once complete, the Voronoi sets are updated by assigning $x$ to $V_i$ if
$$||(I_m - U_iU_i^T)(x - \bar{x}_i)||_2^2 < ||(I_m  - U_jU_j^T)(x - \bar{x}_j)||_2^2,\; i \neq j.$$
 The \textsf{CVOD} algorithm can be modified for \textsf{VQPCA} by replacing $z_i\leftarrow 0$ with $z_i \leftarrow \bar{x}_i$ and $V_i$ with $V_i - \bar{x}_i$ in the call to \textsf{UpdateCentroidsFixed}, and using $(I_m - U_iU_i^T)(x - \bar{x}_i)$ instead of $(I_m - U_iU_i^T)x$ in the expression for $\mathcal{G}^j$.







\floatname{algorithm}{}
\begin{algorithm}[H]
\renewcommand{\thealgorithm}{}
\caption{\centering \textbf{Subroutine: }\textsf{FindVoronoiSets}}
\label{alg:findvoronoisets}
\KwData{A data matrix $A \in \mathbb{R}^{m \times n}$, with rank($A$) = $\rho$, a set of generalized centroids, $\{U_i\}_{i=1}^k,$ and a collection of mean vectors, $\{z_i\}_{i=1}^k$.}
\KwResult{$\{V_i\}_{i=1}^k$, where the $V_i$ form an updated partition of the columns of $A$.}


\vspace{1.5\baselineskip}

$\Omega \leftarrow $set of column vectors of $A$\\
$k \leftarrow $Number of centroids, $U_i$\\
$V_i \leftarrow \emptyset,\; i = 1,\ldots,k$\\

\For{$ x \in \Omega$}{
    \For{$i = 1,\ldots,k$}{
       $d_i \leftarrow \left | \left | (x - z_i) - U_iU_i^T(x - z_i) \right | \right |_2^2$\\
    }
    Assign $x$ to $V_i$ with $d_i < d_j\; i \neq j$\\
}

return $\{V_i\}_{i=1}^k$\\
\end{algorithm}

\section{Novel Adaptive Partitioning Strategies}
\subsection{Adaptive CVOD}

In this section we present a modification of the $\mathsf{CVOD}$ algorithm.  The problem to be solved is given by
$$\min_{\{(V_i,\Theta_i)\}_{i=1}^k}\mathcal{G}_1\quad \mbox{such that}$$
$$\Theta_i^2 = \Theta_i,\quad \sum_{i=1}^k\mbox{rank}(\Theta_i) = r,\quad \bigcup_{i=1}^kV_i = \Omega_A,$$
where $r$ is the target rank parameter for the CSSP problem.  Although the energy functional is the same as that for \textsf{CVOD}, $\mathsf{adaptCVOD}$ uses a different constraint on the centroids.  This reflects a more global approach to reducing the value of $\mathcal{G}_1$ at each iteration. Holding the $V_i$ fixed, the $\mathsf{CVOD}$ algorithm reduces the value of $\mathcal{G}_1$ by determining the optimal projector for each $V_i$, a process that is local in nature.  In other words, for fixed $V_i$, $\mathsf{CVOD}$ minimizes the following:
$$\sum_{i=1}^k \sum_{x \in V_i} ||(I_m - \Theta_i)x||_2^2 =  \sum_{i=1}^k||(I_m - \Theta_i)V_i||_F^2$$
over projectors $\Theta_i$ of $\mbox{rank}(\Theta_i) = d_i$.  The $\mathsf{adaptCVOD}$ algorithm solves this expression from a more global standpoint.  This is done by solving
\begin{eqnarray*}
    \min_{\Phi}||\mathsf{diag}(V_i) - \Phi \mathsf{diag}(V_i)||^2_F&s.t.& \Phi^2 = \Phi \in \mathbb{R}^{km \times km}\\
    & & \mbox{rank}(\Phi) = r.\\
\end{eqnarray*} 
The solution is given by $\mathsf{diag}(U_iU_i^T)$ where each $U_i \in \mathbb{R}^{m \times d_i}$ contains the top $d_i$ left singular vectors of $V_i$.  The rest of the alternating minimization process is the same as that in $\mathsf{CVOD}$.  

\begin{algorithm}[H]
\label{updatecentroidsfixed}
\renewcommand{\thealgorithm}{}
\caption{\centering \textbf{Subroutine: }\textsf{UpdateCentroidsFixed}}

\KwData{A column partition, $\{Y_i\}_{i=1}^k$ of a matrix $A \in \mathbb{R}^{m \times n}$ with rank($A$) = $\rho$ and a multi-index $d = (d_1 \dots d_k)$.}
\KwResult{$\{U_i\}_{i=1}^k$, $k$, where the $U_i$ form an updated set of $k$ generalized centroids and $k$ is the number of Voronoi sets.}


\vspace{1.5\baselineskip}

\For{$i = 1,\ldots,k$}{
    $\tilde{U}\Sigma W^T \leftarrow \mathsf{SVD}(Y_i)$\\
    $U_i \leftarrow \tilde{U}(:,1:d_i)\\$
}

return $(\{U_i\}_{i=1}^k,k)$\\
\end{algorithm}

When updating the centroids in $\mathsf{adaptCVOD}$, it may happen that no left singular vector from one or more of the $V_i$ contributes to the dominant $r$-dimensional subspace of $\mathsf{diag}(V_i)$.  For example, the rank $r$ left singular matrix of $\mathsf{diag}(V_i)$ could look like the following:
$$\left ( \begin{array}{cccc}
          U_1U_1^T & & &\\
           & \ddots & &\\
           & & U_{k-1}U_{k-1}^T & \\
           & & & 0\\
           \end{array}\right )$$

In this case, we allow the number of sets, $k$, to change in order to match the number of singular matrices from each of the $V_i$ that contribute to the rank $r$ SVD of $\mathsf{diag}(V_i)$; see subroutine \textsf{UpdateCentroidsAdapt}. The result is a data driven routine in which the dimension and number of the sets $V_i$ are allowed to vary over the course of the algorithm.  The resulting pseudocode is given by replacing the \textsf{UpdateCentroidsFixed} routine with \textsf{UpdateCentroidsAdapt} in the \textsf{CVOD} algorithm.

We remark that, although the algorithm is motivated by the work in \cite{du2003centroidal}, the inclusion of the $\mathsf{UpdateCentroidsAdapt}$ subroutine gives it a flavor of some projective clustering approaches (e.g., \cite{agarwal2004k} ).

\subsection{Adaptive VQPCA}
In this section, we extend our adaptive process to the \textsf{VQPCA} algorithm.  The energy functional remains the same,
$$\mathcal{G}_2 = \sum_{i=1}^k \sum_{x \in V_i}||\beta_i - (I_m - \Theta_i)x||_2^2,$$
and our minimization problem becomes:
$$\min_{\{(V_i,\Theta_i,\beta_i)\}_{i=1}^k} \mathcal{G}_2\quad \mbox{such that}$$
$$\Theta_i^2 = \Theta_i,\; \sum_{i=1}^k \mbox{rank}(\Theta_i) = r,\quad \beta_i \in \mathbb{R}^m,\quad \bigcup_{i=1}^k V_i = \Omega_A.$$
We note that the function $\mathcal{G}_2$ is reminiscent of multi-dimensional scaling \cite{hastie2009elements}. A similar minimization problem is used in \cite{kerschen2002non}, \cite{kerschen2005distortion}, but does not permit the Voronoi sets to adapt to the data.

As before, the alternating minimization process is repeated until the relative change in $\mathcal{G}_2$ between iterations falls below some chosen $\epsilon>0$.  The $\mathsf{CVOD}$ pseudocode can be modified to suit this process by making the following changes: 
\begin{enumerate}
    \item Replace $z_i \leftarrow 0$ with $z_i \leftarrow \bar{x}_i$, $i = 1,\ldots,k.$
    \item Replace \textsf{UpdateCentroidsFixed}   with \textsf{UpdateCentroidsAdapt}.
    \item Set the inputs for \textsf{UpdateCentroidsAdapt} to be $\{V_i - \bar{x}_i\}_{i=1}^k,\; r$.
    \item Replace $(I_m - U_iU_i^T)x$ with $(I_m - U_iU_i^T)(x - \bar{x}_i)$ in the expression for $\mathcal{G}^j$.
\end{enumerate}

\begin{algorithm}[H]
\label{alg:updatecentroidsadapt}
\renewcommand{\thealgorithm}{}
\caption{\centering \textbf{Subroutine: }$\mathsf{UpdateCentroidsAdapt}$}
\KwData{A column partition, $\{Y_i\}_{i=1}^k$ of a matrix $A \in \mathbb{R}^{m \times n}$, with rank($A$) = $\rho$, and a positive integer $r \le \rho$.}
\KwResult{$(\{U_i\}_{i=1}^{\tilde{k}},\tilde{k})$, where the $U_i$ form an updated set of $\tilde{k}$ generalized centroids}


\vspace{1.5\baselineskip}

\For{$i = 1,\ldots,k$}{
    $U_i^{(0)}\Sigma_{i}^{(0)}W_i^{(0)} \leftarrow \mathsf{SVD}(Y_i)$\\
    $S_i \leftarrow \mbox{ singular values of }\Sigma_i$\\
    $U_i^{(1)} \leftarrow \emptyset$\\
}
$S \leftarrow \mbox{Top r singular values of } diag(\Sigma_i)$\\

$\tilde{k} \leftarrow 0\\$
\For{$\sigma \in S$}{
    \For{$i = 1,\ldots,k$}{
        \If{$\sigma \in S_i$}{
            $U_i^{(1)} \leftarrow $Append corresponding column from $U_i^{(0)}$\\
            $\tilde{k} \leftarrow \tilde{k} + 1\\$
        }
    }
}

return $ \left ( \{U_i^{(1)}\}_{i=1}^{\tilde{k}},\tilde{k} \right)$\\
\end{algorithm}

\section{Partitioned DEIM with Adaptive Column Selection}
The result of each of the processes outlined in the previous section is a collection, $\{V_i, U_i\}_{i=1}^k$, which describes a partition of the columns of $A$ as well as low dimensional representations of each member in the partition.  The task now is to apply an existing CSSP algorithm to each $V_i$ and return a combined result.  Although nearly any CSSP algorithm will work, we use the \textsf{DEIM} \cite{sorensen2016deim} algorithm as an example. The result will be four algorithms that each combine \textsf{DEIM} with one of the column-partitioning algorithms from the previous section: \textsf{CVOD+DEIM}, \textsf{VQPCA+DEIM}, \textsf{adaptCVOD+DEIM}, and \textsf{adaptVQPCA+DEIM}. The last two represent new, adaptive routines for the column-selection task.

As mentioned earlier, to select $r$ columns from a matrix $A$ with rank $\rho$, the \textsf{DEIM} algorithm requires a target rank parameter $0<r\le \rho$ and a full column rank matrix $W \in \mathbb{R}^{n \times r}$, which usually consists of the top $r$ right singular vectors of $A$.  We will represent this operation by $\mathsf{DEIM}(A,r,W)=C$, where $C \in \mathbb{R}^{m \times r}$.  From \cite{sorensen2016deim}, we know \textsf{DEIM} returns linearly independent columns.  However, it is unclear if the cumulative result from applying \textsf{DEIM} to each $V_i$ will be a linearly independent set.  To ensure that this is the case, we perform the following (See the \textsf{PartionedDEIM} algorithm).  First, we sort the $V_i$ in ascending order by the ranks of their centroids; i.e, 
$$\{V_1,\ldots,V_k\} \iff \mbox{rank}(U_i) \le \mbox{rank}(U_{i+1}).$$

\floatname{algorithm}{}
\renewcommand{\algorithmcfname}{}
\begin{algorithm}[H]
\renewcommand{\thealgorithm}{}
\caption{\centering \textbf{Algorithm:} $\mathsf{Partitioned DEIM}$}
\label{alg:pDEIM}
\KwData{A column partition, $\{V_i\}_{i=1}^k$ of a matrix $A \in \mathbb{R}^{m \times n}$, with rank($A$) = $\rho$, a positive integer $r<\rho$, and a collection, $\{U_i\}_{i=1}^k,$ of $m \times d_i$ matrices containing the top $d_i$ left singular vectors of each $V_i$ with $\sum_{i=1}^k d_i = r.$}
\KwResult{$C \in \mathbb{R}^{m \times \tilde{r}}$, $\tilde{r} \le r$, such that $||A - CC^\dagger A||_F$ is small.}


\vspace{1.5\baselineskip}

$\{V_i\}_{i=1}^k \leftarrow $ Sort $V_i$ by $\mbox{rank}(U_i) \le \mbox{rank}(U_{i+1})$\\
$C_1 \leftarrow \mbox{Select }d_1 \mbox{ columns from } V_1 \mbox{ via DEIM}$\\
\hspace{0.4\baselineskip} $C \leftarrow C_1$\\

\For{$i = 2,\ldots,k$}{
\hspace{0.4\baselineskip}$QR \leftarrow \mbox{qr}(C)$ \tcp*{QR-decomposition}
    \hspace{0.4\baselineskip}$\tilde{V}_i\leftarrow (I - QQ^T)V_i$\\
    $\tilde{U}\tilde{\Sigma}\tilde{W}^T \leftarrow \mathsf{SVD}(( I - QQ^T)V_i)$\\
    $d_i \leftarrow \mbox{rank}(U_i)$\\
    $C_i \leftarrow \mathsf{DEIM}(V_i,d_i,\tilde{W}(:,1:d_i))$\\
    $C \leftarrow [C_1 \dots C_i]$\\
}
return $C$\\
\end{algorithm}

Next, we run \textsf{DEIM} on $V_1$ with $V_1^TU_1$ to select $d_1 = \mbox{rank}(U_1)$ columns from $v_1$:
$$C = \mathsf{DEIM}(V_1,d_1,V_1^TU_1).$$
Note that $V_1^TU_1$ is a scaled version of the matrix containing the top $d_1$ right singular vectors of $V_1$.  To select columns from $V_2$, we first compute the SVD of the matrix
$$(I - CC^\dagger)V_2 = \tilde{U}\tilde{\Sigma}\tilde{W}^T.$$
We then select $d_2 = \mbox{rank}(U_2)$ columns from $V_2$ with \textsf{DEIM} using $\tilde{W}(:,1:d_2)$; i.e.,  $\mathsf{DEIM}(V_2,d_2,\tilde{W}(:,1:d_2))$.  The resulting columns are appended to the matrix $C$ and the process repeats until $C$ has $r$ columns.  As shown by Lemma \ref{lemma1}, the final matrix $C \in \mathbb{R}^{m \times r}$ will have full column rank.

\begin{lemma}\label{lemma1}
    Let $A \in \mathbb{R}^{m \times n}$ with $\mbox{rank}(A) = \rho$, and let $C \in \mathbb{R}^{m \times r}$ be the result from applying \textsf{PartionedDEIM} to the output from any of the previously discussed partitioning algorithms. Then $C$ has full column rank.
\end{lemma}

\begin{proof} 
    Let $(\{V_i\}_{i=1}^{k},\{U_i\}_{i=1}^k)$ be in the input to \textsf{PartionedDEIM}, where the $V_i$ denote the Voronoi sets and $U_i$ are the centroids.  Define $d_i = \mbox{rank}(U_i), i =1,\ldots,k$, and assume the Voronoi sets have been ordered as in  \textsf{PartitionDEIM}.    We may write $C = [C_1 \dots C_k],$ where $C_i \in \mathbb{R}^{m \times d_i}$ contains those columns of $C$ that belong to $V_i$.  Since $C_1$ results from applying \textsf{DEIM} to $V_1$, we know that it has full column rank \cite{sorensen2016deim}.  Proceeding by induction, suppose $C = [C_1 \dots C_s]$, $1<s<k$ has been constructed and has full column rank.  We next consider $B = (I - QQ^T)V_{s+1} \in \mathbb{R}^{m \times n_{s+1}}$, where $QR = C$ is the QR-decomposition of $C$.  Let $T_{s+1} \in \mathbb{R}^{n_{s+1} \times d_{s+1}}$ be the \textsf{DEIM} selection matrix.  Then $BT_{s+1}$ has full column rank, and each column is linearly independent with respect to the columns of $C$.  Now suppose that $V_{s+1}T_{s+1}$ does not have full column rank.  Then there exists $x \neq 0$ in $\mathbb{R}^{n_{s+1}}$ such that $V_{s+1}T_{s+1}x = 0$.  But this implies
    $$||Bx||_2 = ||(I - QQ^T)V_{s+1}x||_2 \le ||V_{s+1}T_{s+1}x||_2 = 0,$$
    a contradiction.  Thus, the $V_{s+1}T_{s+1}$ has full column rank
\end{proof}

\section{Numerical Experiments}
In this section we investigate the performance of the partitioning/\textsf{DEIM} algorithms (\textsf{CVOD+DEIM}, \textsf{VQPCA+DEIM}, \textsf{adaptCVOD+DEIM}, and \textsf{adaptVQPCA+DEIM}) on three data sets using the original \textsf{DEIM} as a benchmark.

The first data set, referred to hereafter as SNN1E3, uses sparse non-negative matrices (SNN) of the form:
$$A = \sum_{i=1}^l\frac{2}{i}x_iy_i^T + \sum_{i=l+1}^n\frac{1}{i}x_iy_i^T,$$
where $x_i \in \mathbb{R}^m$, $y_i \in \mathbb{R}^n$ are random sparse vectors generated via Matlab's \cite{MATLAB} 'sprand' command. We use the following paramters:
$$m = n = 1000,\quad l = 100,\; \mbox{density} = 0.0125.$$
This last parameter is required by 'sprand', and controls the sparseness of the output.  We remark that similar test data is used in \cite{dong2021simpler} and \cite{sorensen2016deim}.
The second data set is $A \in \mathbb{R}^{60000 \times 784}$ and consists of MNIST training data with images set as rows.  We refer to this data set as MNIST.

\begin{figure}[H]
\centering
\includegraphics[width=.6\linewidth]{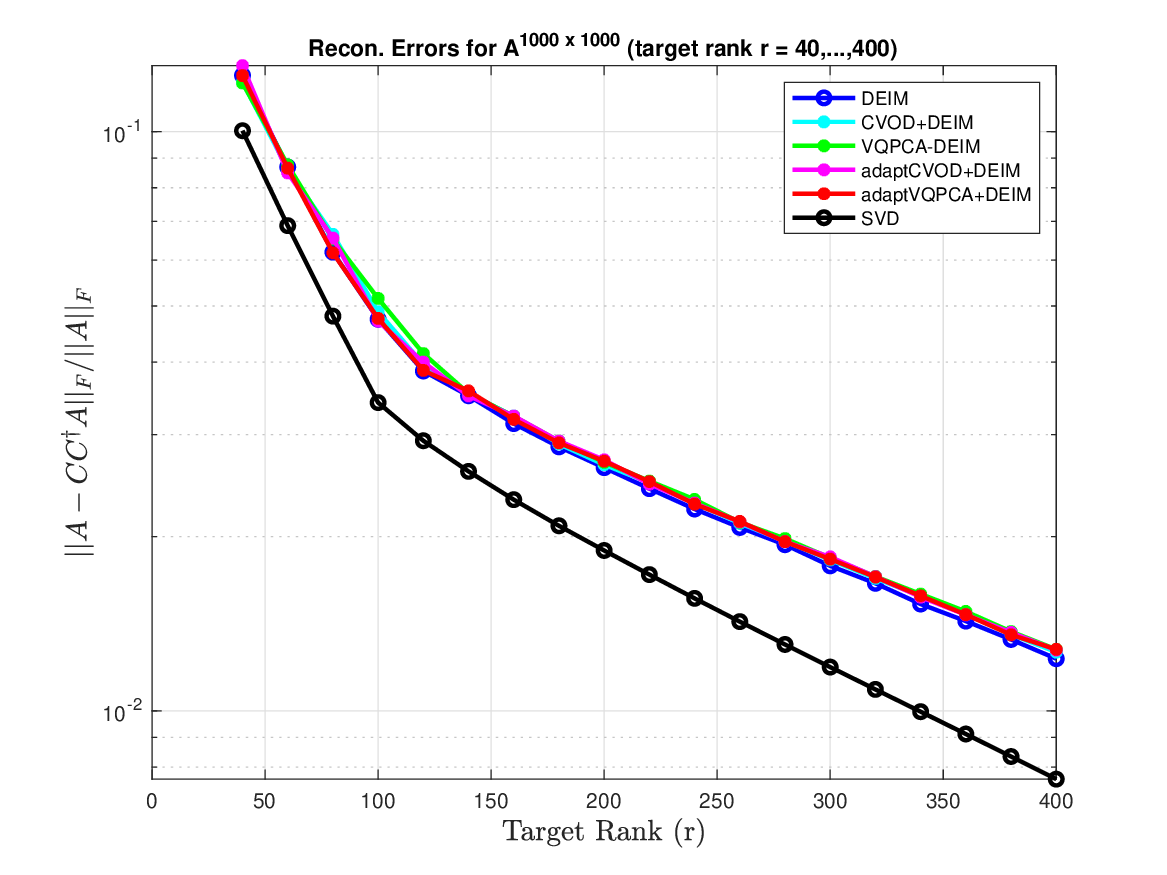}
\caption{Normalized reconstruction errors for \text{DEIM} with and without partitioning on the sparse SNN1E3 data.}
\label{fig:snn1e3}
\end{figure}

\textbf{Algorithm Settings and Metrics.}  
On the SNN1E3 data set, we consider ranks $r = 40,60,\ldots,400$ with $k = 20$ Voronoi sets.  Tests on the MNIST use ranks $r = 30,40,\ldots,150$ with $k = 5,10$ Voronoi sets.  In all cases, we set the stopping parameter to $\epsilon = 0.1$ and the multi-index values to $d_i = \frac{r_i}{k}$. For each rank, we measure the normalized reconstruction error
$$\frac{||(I - CC^\dagger)A||_F}{||A||_F}.$$

We use the same initial partitions for each of the algorithms.  In addition, we reduce the dimension of each data set for each rank, $r$, by left-multiplying $A\in \mathbb{R}^{m \times n}$ by an $r \times m$ matrix consisting of independent Gaussians with mean zero and standard deviation $r^{-1/2}$ \cite{dong2021simpler}.  
Our results are shown in Figures \ref{fig:snn1e3}-\ref{fig:mnist}.

\section{Analysis and Discussion}
Figures \ref{fig:snn1e3} - \ref{fig:mnist} show that, in general, the performance of the partitioned-based \textsf{DEIM} algorithms are on par, and sometimes better, than the original \textsf{DEIM} algorithm. In addition, it appears that the size of the partition used (i.e., the number of Voronoi sets) has a weak effect on the reconstruction errors.  This suggests benefits are to be had from using a partition-based alternative to traditional column-selection/CUR decomposition algorithms on large data sets \cite{bui2005model}; e.g., parallelization and handling of large clusters \cite{zdybal2022advancing}. Though somewhat counter-intuitive, this phenomenon is expected given the following.

\begin{figure}[H]
\centering
\includegraphics[width=.6\linewidth]{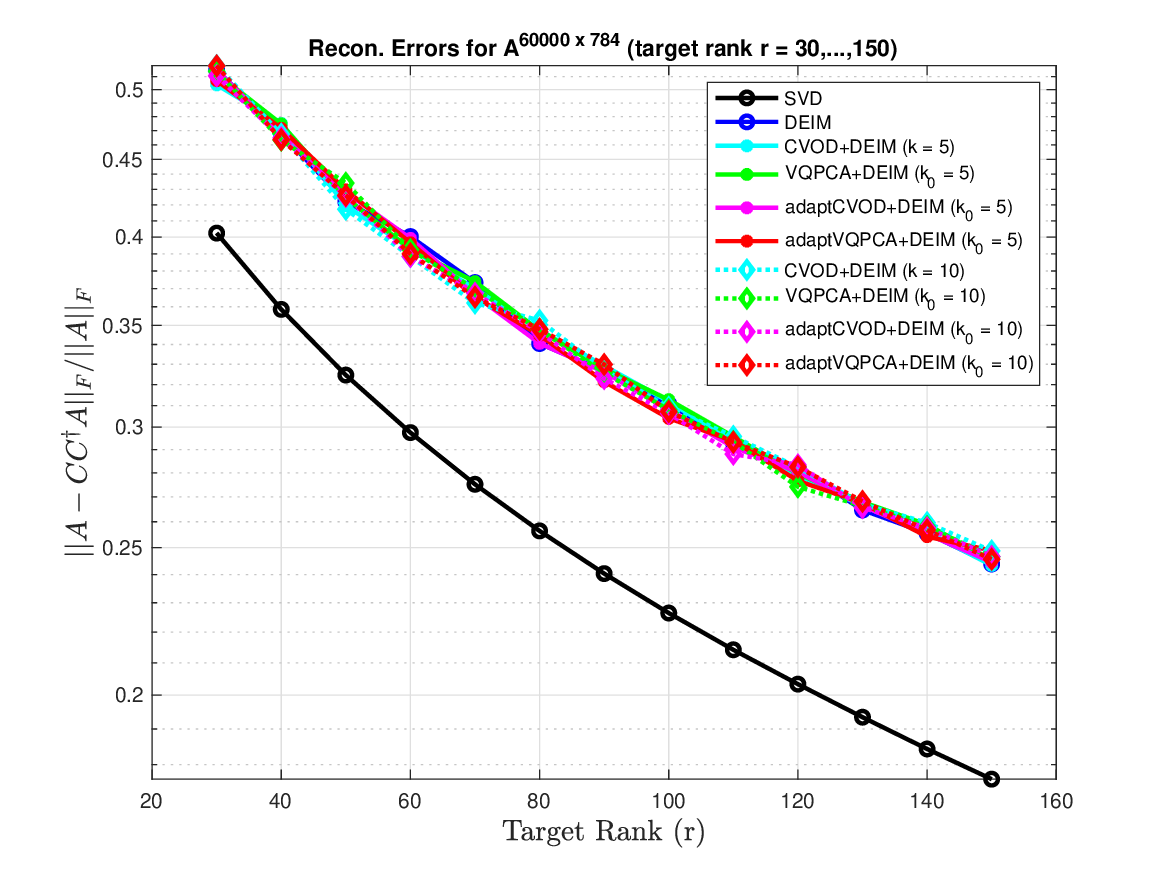}
\caption{Normalized reconstruction errors for \text{DEIM} with and without partitioning on the MNIST data set using $k = 5,10$ Voronoi sets.}
\label{fig:mnist}
\end{figure}

\begin{lemma}\label{lemma2}
Let $A \in \mathbb{R}^{m \times n}$ with $\mbox{rank}(A) = \rho$, and let $0 < r < \rho$ be a desired target rank.  Let $C \in \mathbb{R}^{m \times r}$ be the matrix resulting from any of the partition-based \textsf{DEIM} algorithms with an initial column partition of size $k$ and multi-index $d = (d_1 \dots d_k)$, with $d_i = \lfloor r/k \rfloor$.  
If $\{V_i\}_{i=1}^{\tilde{k}}$ is the final column partition with $\tilde{k}\le k$, then
$$||(I - CC^\dagger)A||_F\le \sqrt{2\tilde{k}\gamma_C}||A - A_r||_F,$$
    where $\gamma_C = \max_i||(I - C_iC_i^\dagger)V_i||_F^2$, $C_i \in \mathbb{R}^{m \times \tilde{d}_i}$ contains the columns of $C$ selected from $V_i$, and $A_r \in \mathbb{R}^{m \times n}$ denotes the best rank $r$ approximation to $A$ given by the truncated SVD.
\end{lemma}
\begin{proof}
Let $A$, $C$, and $\{V_i\}_{i=1}^{\tilde{k}}$ be given as in the statement of the lemma.  Write $C = [C_1 \dots C_{\tilde{k}}]$, where the columns of $C_i \in \mathbb{R}^{m \times \tilde{d}_i}$ are those from $C$ that belong to $V_i$. Let $A_r \in \mathbb{R}^{m \times n}$ denote the best rank $r$ approximation to $A$ given by the truncated SVD, and define $\mathcal{E}_r = ||A - A_r||_F$.  We have
\begin{eqnarray*}
    ||(I - CC^\dagger)A||_F^2 &=& \sum_{i=1}^{\tilde{k}}||(I-CC^\dagger)V_i||_F^2\\
    &\le& \tilde{k}\max_i ||(I - C_iC_i^\dagger)V_i||_F^2\\
    &\le&\tilde{k}\max_i ||(I - C_iC_i^\dagger)V_i||_F^2(1 + \mathcal{E}^2_r).\\
\end{eqnarray*}
The first inequality follows since $CC^\dagger$ and $C_iC_i^\dagger$ are both orthogonal projectors, and $\mbox{span}(C_iC_i^\dagger) \subset \mbox{span}(CC^\dagger).$ 
This implies
$$\frac{||(I - CC^\dagger)A||_F^2}{\mathcal{E}_r^2} \le 2\tilde{k}\max_i ||(I - C_iC_i^\dagger)V_i||_F^2.$$

If we let $\gamma_C = \max_i ||(I - C_iC_i^\dagger)V_i||_F^2$, a value related to the worst local approximation of all the Voronoi sets, multiplying both sides by $\mathcal{E}_r^2$ and taking square roots gives the desired result. 
\end{proof}

Thus, the errors (worst case) increase sublinearly with the size of the final partition, $\tilde{k}$.  We remark that this argument is not restricted \textsf{DEIM} algorithm; it is applicable to \emph{any} column subset selection algorithm.  In terms of a CUR decomposition, the previous implies the following.

\begin{theorem*}
 Let $A \in \mathbb{R}^{m \times n}$ with $\mbox{rank}(A) = \rho$, and let $0 < r < \rho$ be a desired target rank.  Suppose $C \in \mathbb{R}^{m \times r}$ and $R \in \mathbb{R}^{r \times n}$ are the result from applying any of the partition-based \textsf{DEIM} algorithms on $A$ and $A^T$ respectively, each with an initial partition of size $k$ and multi-index defined as in Lemma \ref{lemma2}.  If $\{V_i\}_{i=1}^{k_1}$ and $\{W_j\}_{j=1}^{k_2}$ denote the respective final column and row partitions with $k_1,k_2 \le k$, then
 $$||A - CUR||_F \le\left ( \sqrt{2k_1\gamma_C} + \sqrt{2k_2\gamma_R} \right ) ||A - A_r||_F,$$
 where
 $$\gamma_C = \max_i ||(I - C_iC_i^\dagger)V_i||_F^2,\quad \gamma_R = \max_i ||W_j(I - R_j^\dagger R_j)||_F^2$$
 are from Lemma \ref{lemma2} and $A_r \in \mathbb{R}^{m \times n}$ denotes the best rank $r$ approximation to $A$ given by the truncated SVD.
\end{theorem*}

\begin{proof}
   We have
   \begin{eqnarray*}
||A - CUR||_F &=& ||A - CC^\dagger A R^\dagger R||_F\\
     &\le& ||(I - CC^\dagger)A||_F + ||CC^\dagger A(I - R^\dagger R)||_F\\
    &\le & \left ( \sqrt{2k_1\gamma_C} + \sqrt{2k_2\gamma_R} \right ) ||A - A_r||_F,\\
\end{eqnarray*}
where $\gamma_R$ is analogous to $\gamma_C$ but results from processing the rows of $A$.
\end{proof}

\begin{figure}[H]
\centering
\includegraphics[width=.6\linewidth]{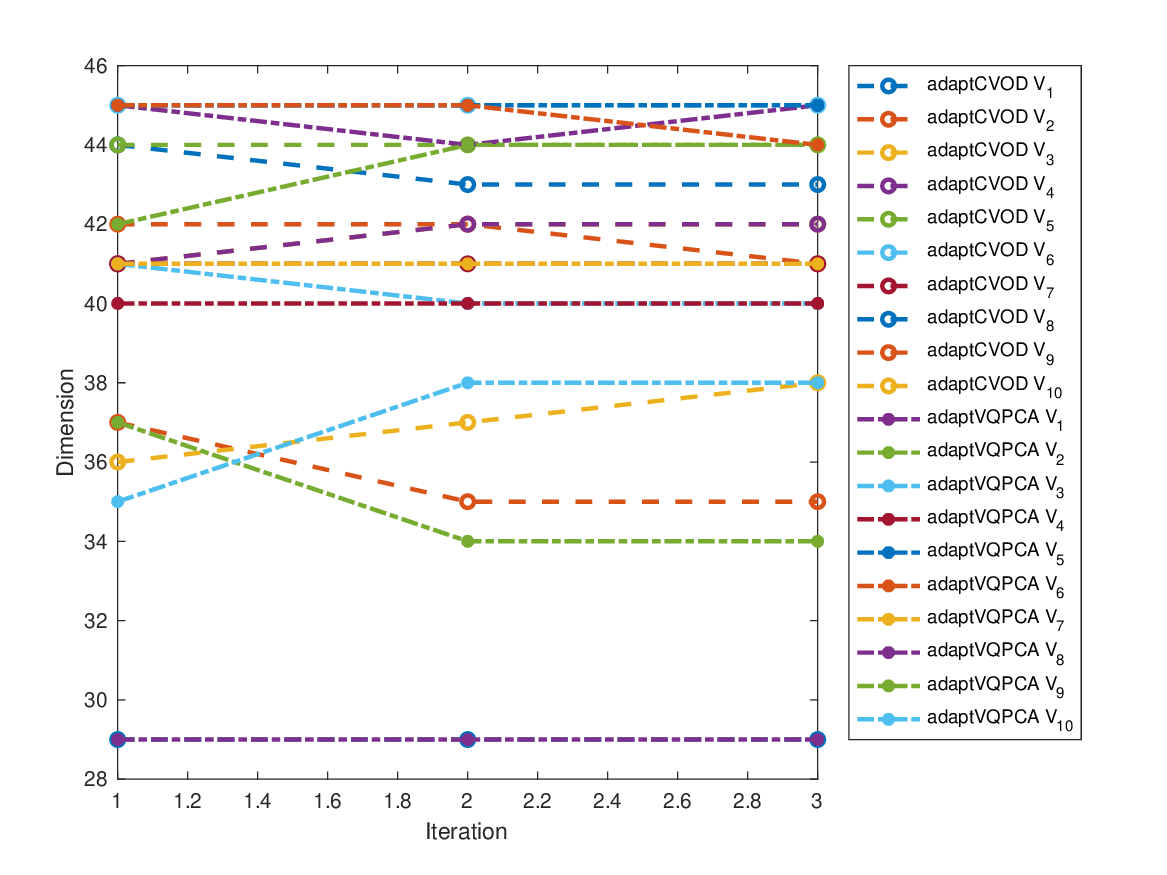}
\caption{Evolution of the centroid dimensions for each Voronoi set resulting from \textsf{adaptCVOD} and \textsf{adaptVQPCA} applied to the MNIST data.  In both cases, $k = 10$, $\mbox{rank} = 400$, and $\epsilon = 0.01$. }
\label{fig:evolvingDims}
\end{figure}


\begin{figure}[H]
\centering
\includegraphics[width=.6\linewidth]{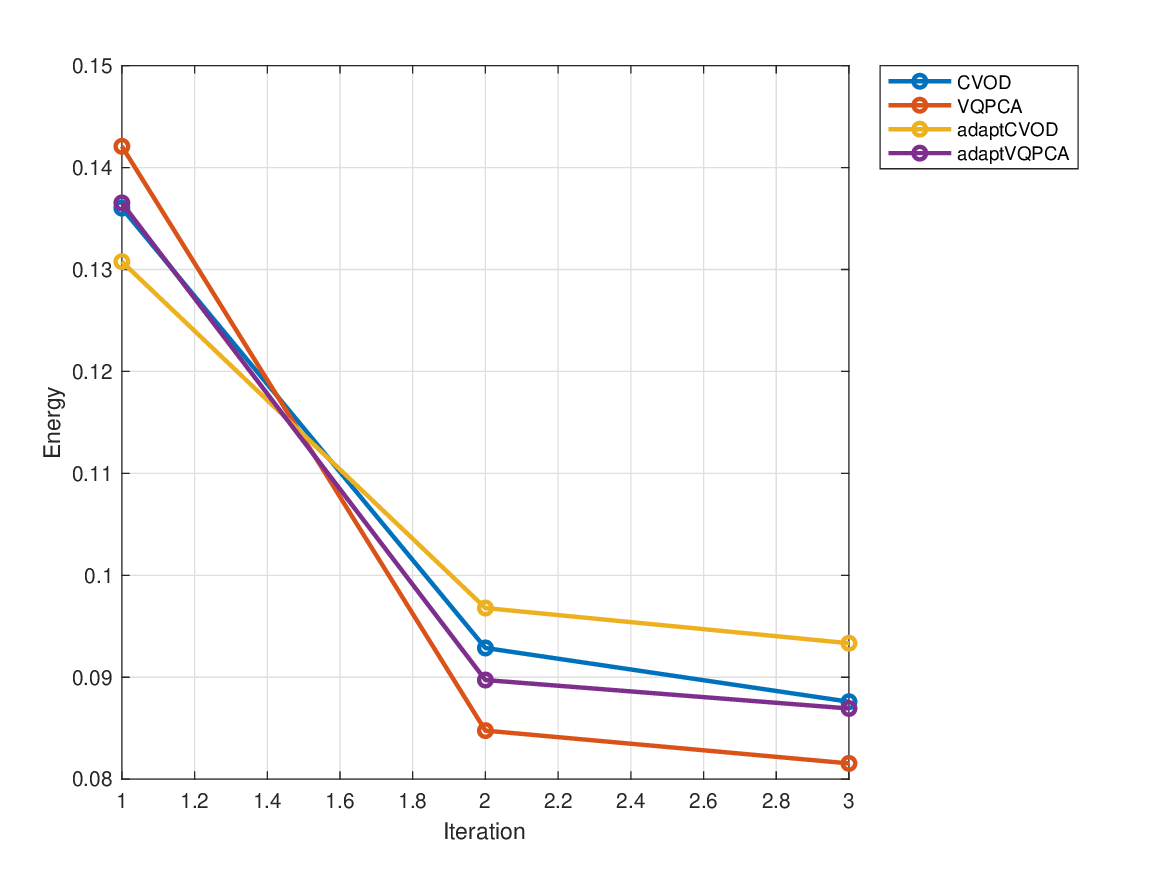}
\caption{Energy profiles from all of the partition-based algorithms on the MNIST data set.  The parameters used are $k = 10$, $\mbox{rank} = 400$, and $\epsilon = 0.01$. }
\label{fig:MNISTinfo}
\end{figure}


		
		

Here, we note that, in its present form, the dependence of Lemma \ref{lemma2} on the choice of partitioning algorithm is only somewhat apparent through the $\gamma_C$ term, but is still coupled with \textsf{DEIM}.  Our investigation into the relationship between the choice of partitioning algorithm and the resulting CSSP solution will be treated in another paper.

In addition to exhibiting similar reconstruction errors, the partition-based algorithms also possess similar energy (objective function) profiles.  Figure \ref{fig:energies} shows the energy profiles for each of the algorithms on the MNIST data set with $k = 10$, $\mbox{rank} = 400$, and $\epsilon = 0.01$.  We remark that although the \textsf{VQPCA}-based routines attain the lowest final energies, it is possible that the resulting partitions are less informative than their \textsf{CVOD} counterparts \cite{tadavani2012low}.  Figure \ref{fig:evolvingDims} shows how the Voronoi set dimensions evolve on the MNIST data with the same parameters.  This type of data-driven process could be helpful in revealing hidden information in the data as well as highlight data resulting from different processes \cite{zdybal2022advancing}.

For an $m \times n$ matrix, target rank parameter, $r$, and $k$ initial Voronoi sets, each of the partitioned-based \textsf{DEIM} algorithms (including the post-processing computations) scale as
$$\mathcal{O}(nmr + kmn_{max}^2).$$
Here, $n_{max}$ refers to the largest cardinality of the Voronoi sets.  Applying a sketching matrix of size $r \times m$ reduces this to
$$\mathcal{O}(k\min\{rn_{max}^2,r^2n_{max}\} + nr^2),$$
which does not include the cost of constructing and applying the sketch.  Both expressions are on par with a number of existing column selection algorithms; see \cite{dong2021simpler} for a listing that includes the costs of sketching.

\section{Conclusion}
In this article we describe a general process for reducing the column-subset selection problem (CSSP) to a collection of smaller sub-problems.  This is accomplished by first applying a partitioning algorithm to the columns of the matrix in question.  For this task, we use the \textsf{CVOD} and \textsf{VQPCA} algorithms, as well as adaptive extensions of these routines.   Referred to as \textsf{adaptCVOD} and \textsf{adaptVQPCA}, these last algorithms use a data-driven process to select the number of sets included in the partition, and the dimensions of each set.  The second step is to apply an existing CSSP algorithm to each partition and combine the results; we consider the \textsf{DEIM} algorithm in this article.  We show theoretically and empirically that the resulting column selection and CUR decomposition solutions are competitive with \textsf{DEIM} in terms of accuracy and complexity (without considering parallelizaton).  The partitioning algorithms require few input parameters, and enjoy worst case error bounds that scale as $\sqrt{k}$ for a fixed problem, where $k$ is the number of Voronoi sets. The result from combining CSSP and partitioning algorithms is a process that is well-suited for parallelization.  This can be advantageous when encountering large data sets, especially those for which computation of an SVD is prohibitive.  Potentially attractive uses include weather 
applications (e.g., data assimilation) where such data sets are common.
Our future work includes optimizing the current reconstruction error bounds, applying our results to model order reduction applications, and examining the clustering ability of different partitioning algorithms.  In particular, we intend to examine other so-called distortion measures used to assign points to their respective Voronoi sets.  Lastly, we plan to investigate further the relationship between the optimality properties of the partitioning algorithms and the resulting column-selection solutions.






\bibliographystyle{unsrt}
\bibliography{references}

\end{document}